\documentclass[twoside,reqno,12pt]{amsart}
\usepackage{latexsym}
\usepackage{amssymb,amsmath,amsopn,amsmath}
\usepackage{enumerate}

\newtheorem{thm}{Theorem}
\newtheorem{lem}{Lemma}

\renewcommand{\Re}{\mathbb R}
\newcommand{\Ze}{\mathbb Z}
\newcommand{\Red}{\Re^d}

\newcommand{\ReD}{\Re^{D}}

\newcommand{\Int}{\int\limits}
\newcommand{\dx}{\diff\! x}

\DeclareMathOperator{\conv}{conv}
\DeclareMathOperator{\inter}{int}
\DeclareMathOperator{\vol}{V}
\DeclareMathOperator{\diff}{d}
\DeclareMathOperator{\interior}{int}
\newcommand{\bd}{\operatorname{bd}}
\newcommand{\setbuilder}[2]{\{#1:#2\}}
\newcommand{\st}{\; : \; }

\newcommand{\numbersystem}[1]{\mathbb{#1}}
\newcommand{\fhi}{\varphi}
\newcommand{\bR}{\numbersystem{R}}
\newcommand{\norm}[1]{\left\lVert#1\right\rVert}
\newcommand{\abs}[1]{\left\lvert#1\right\rvert}
\newcommand{\card}[1]{\left\lvert#1\right\rvert}

\begin{document}

\title{Arrangements of homothets of a convex body II}
\author[M. Nasz\'odi, K. J. Swanepoel]{M\'arton Nasz\'odi \and Konrad J. Swanepoel}
\address{Department of Geometry, Lorand E\"otv\"os University, Pazm\'any 
P\'eter S\'etany 1/C Budapest, Hungary 1117}\email{marton.naszodi@math.elte.hu}
\address{Department of Mathematics, London School of Economics and Political 
Science, Houghton Street, London WC2A 2AE, United 
Kingdom}\email{k.swanepoel@lse.ac.uk}
\begin{abstract}
A family 
of homothets of an $o$-symmetric convex body $K$ in $d$-dimensional Euclidean space is called a Minkowski arrangement if no homothet contains the center of any other homothet in its interior.
We show that any pairwise intersecting 
Minkowski arrangement of a $d$-dimensional convex body has at most $2\cdot 3^d$ members.
This improves a result of Polyanskii (Discrete Mathematics \textbf{340} (2017), 1950--1956).
Using similar ideas, we also give a proof the following result of Polyanskii: 
Let $K_1,\dots,K_n$ be a sequence of homothets of the $o$-symmetric convex body $K$, such that for any $i<j$, the center of $K_j$ lies on the boundary of $K_i$.
Then $n\leq O(3^d d)$.
\end{abstract}
\maketitle

\section{Introduction}
We use the notation $[n]=\{1,2,\dots,n\}$.
A \emph{convex body} $K$ in the $d$-dimensional Euclidean space $\Red$ is a 
compact convex set with non-empty interior, and is \emph{$o$-symmetric} if $K=-K$.
A (positive) \emph{homothet} of $K$ is a set 
of the form $\lambda K+v:=\setbuilder{\lambda k+v}{k\in K}$, where $\lambda>0$ is the 
homothety ratio, and $v\in\Red$ is a translation vector.
If $K$ is $o$-symmetric, we also call $v$ the \emph{center} of the homothet $\lambda K+v$.
An \emph{arrangement of homothets of $K$} is a collection $\setbuilder{\lambda_i K+v_i}{i\in[n]}$.
A \emph{Minkowski arrangement} of an $o$-symmetric convex body $K$ is a family $\{v_i+\lambda_i K\}$ 
of homothets of $K$ such that none of the homothets contains the center of any other homothet in its interior.
This notion was introduced by L. Fejes T\'oth \cite{FT65} in the context of Minkowski's fundamental theorem on the minimal determinant of a packing lattice for a symmetric convex body, and was further studied by him in \cite{FT67, FT99}, by B\"or\"oczky and Szab\'o in \cite{BSz04}, and in connection with the Besicovitch covering theorem by F\"uredi and Loeb \cite{FL94}.
Recently, Minkowski arrangements have been used to study a problem arising in the design of wireless networks~\cite{NSS17}.
In \cite{NPS17} it was shown that the largest cardinality of a pairwise intersecting Minkowski arrangement of homothets of an $o$-symmetric convex body in $\Red$ is $O(3^d d\log d)$.
This was improved to $3^{d+1}$ by Polyanskii \cite{P17}.
We make the following slight improvement.
\begin{thm}\label{thm:sashaimproved}
For any $o$-symmetric convex body $K$ in $\Red$, a pairwise intersecting 
Minkowski arrangement has at most $2\cdot 3^d$ members.
\end{thm}
Note that the $d$-cube has $3^d$ pairwise intersecting translates that form a Minkowski arrangement.
The proof uses ideas from \cite{N06} and \cite{LN09}.

In \cite{NPS17}, bounds on pairwise intersecting Minkowski arrangements were used to give an upper bound of $O(6^d d^2\log d)$ on the length of a sequence of homothets $v_i+\lambda_i K$ of an $o$-symmetric convex body $K$ such that $v_j\in\bd(v_i+\lambda_i K)$ whenever $j> i$.
This bound was improved to $O(3^d d)$ by Polyanskii \cite{P17}.
We use some similar ideas to the proof of Theorem~\ref{thm:sashaimproved} to give a short proof of this result of Polyanskii.
\begin{thm}[Polyanskii \cite{P17}]\label{thm:sasha2}
Let $K$ be an $o$-symmetric convex body, and $v_1,v_2,\ldots,v_n\in\Red$.
Let $\lambda_1,\lambda_2,\ldots,\lambda_n >0$, and assume that for any $1\leq i<j\leq n$ we have $v_j\in \bd(v_i+\lambda_i K)$. Then $n\leq O(3^d d)$.
\end{thm}
The interest in this result is that it gives the upper bound $k^{O(3^d d)}$ to the cardinality of a set in a $d$-dimensional normed space in which only $k$ non-zero distances occur between pairs of points.
This is currently the best known upper bound if $k=\Omega(3^d d)$ (see \cite{S17} for a survey of this problem).

\section{Proof of Theorem~\ref{thm:sashaimproved}}

\begin{thm}\label{thm:lifting}
	Let $d\geq 1$. Suppose that there exists an 
	$o$-symmetric convex body $K$ in $\Red$ which has a pairwise intersecting 
	Minkowski arrangement of $n$ homothets.
	Then there exists a set $\{x_1,\dots,x_n\}$ of $n$ points in $\Re^{d+1}$ such that $o\notin \conv\{x_1,\ldots,x_n\}$, and for any distinct $i,j\in[n]$, $i<j$, there exists a non-zero linear functional $f_{ij}\colon\Re^{d+1}\to\Re$ with
\begin{equation}\label{eq:eq0}
    \abs{f_{ij}(x_k)}\leq \abs{f_{ij}(x_i)-f_{ij}(x_j)}\quad\text{for all $k\in[n]$.}
\end{equation}
\end{thm}
We remark that the converse of the above theorem does not hold.
For a simple counterexample, let $\{x_1,\dots,x_5\}$ be the vertex set of a regular pentagon, with $o$ just outside the pentagon, close to the midpoint of an edge.
It is easy to see that for any pair $x_i, x_j$ of vertices there is a line through $o$ such that the projections $\pi(x_k)$ of the vertices onto the line are all within distance $\abs{\pi(x_i)-\pi(x_j)}$ of $o$.
On the other hand, it is also easy to see that a pairwise intersecting Minkowski arrangement of intervals in $\Re$ can have at most two members.

The above remark is to be contrasted with the equivalence in the following result, which generalizes part of Theorem~1.4 of \cite{LN09}.
\begin{thm}\label{thm:equivalence}
	Given $\lambda\geq1$, and $D\in\Ze, D\geq 1$. Then the following statements 
are equivalent.
	\begin{enumerate}[(i)]
		\item\label{item:equivalncepoints} 
		There exists a set $\{x_1,\ldots,x_n\}$ of $n$ points in $\ReD$, 
such that $o \notin \conv\{x_1,\ldots,x_n\}$, and for any distinct $i,j\in [n], i< j$ there exists a non-zero linear 
functional $f_{ij}:\ReD\to\Re$ with 
\begin{equation}\label{eq:eq1}
 \abs{f_{ij}(x_k)}\leq \frac{\lambda}{2}\abs{f_{ij}(x_i)-f_{ij}(x_j)}\quad\text{for all $k\in[n]$.}
\end{equation}
		\item\label{item:equivalncetranslates}
		There is an $o$-symmetric convex set $L$ in $\ReD$ that has $n$ 
non-overlapping translates $L+t_1,\ldots, L+t_n$, each intersecting 
$(\lambda-1)L$, with $o \notin \conv\{t_1,\ldots,t_n\}$.
\end{enumerate}
\end{thm}

We note that the equivalence between (ii) and (iv) of Theorem~1.4 in \cite{LN09} is exactly the above theorem in the case $\lambda=1$.

\begin{thm}\label{thm:closedonesidedbound}
Let $K$ be an $o$-symmetric convex set in $\ReD$ with $D\geq 2$, and let
$\alpha K+t_1,\ldots,\alpha K+t_n$ be $n$ non-overlapping translates of $\alpha 
K$ with 
$\alpha>0$ such that each translate intersects $K$, and $o \notin 
\inter(\conv\{t_1,\ldots,t_n\})$.
Then
	\begin{equation}\label{eq:Halpha}
		n \leq \frac{(1+2\alpha)^{D-1}(1+3\alpha)}{2\alpha^D}.
	\end{equation}
\end{thm}
This theorem is a slight modification of Theorem~1.5 
of \cite{LN09}. There the translates of $\alpha K$ touch $K$, whereas here 
they 
may overlap with $K$.
Theorem~\ref{thm:closedonesidedbound} is sharp for $\alpha=1$. Indeed, let $K$ 
be the cube $[-1,1]^D$, and consider the $2\cdot3^{D-1}$ translation vectors 
$\{t\in\{-2,0,2\}^D\st t^{(1)}\geq t^{(2)}\}$.

Combining Theorems \ref{thm:lifting}, \ref{thm:equivalence} and 
\ref{thm:closedonesidedbound} (with $\lambda=2$, $K=(\lambda-1)L=L$, 
$\alpha=\frac{1}{\lambda-1}=1$), we immediately obtain 
Theorem~\ref{thm:sashaimproved}.

\section{Proof of Theorem~\ref{thm:lifting}}
Let the Minkowski arrangement by $\setbuilder{v_i+\lambda_i K}{i\in[n]}$, where $\lambda_i>0$ and $v_i\in\Red$ for each $i\in[n]$.
Let $x_i=(\lambda_i^{-1}v_i,\lambda_i^{-1})\in\Red\times\bR$, $i\in[n]$.
Fix distinct $i,j\in\{1,\dots,n\}$.
We will find a linear $f\colon\Red\times\bR\to\bR$ that satisfies \eqref{eq:eq0}.
Let $\fhi\colon\Red\to\bR$ be a linear functional such that 
$\fhi(x)\leq\norm{x}_K$ for all $x\in\Red$ and $\fhi(v_j-v_i)=\norm{v_j-v_i}_K$.
(Thus, $\fhi^{-1}(1)$ is a hyperplane that supports $K$ at 
$\norm{v_j-v_i}_K^{-1}(v_j-v_i)$.)

Since any two homothets $v_k+\lambda_k K$ and $v_\ell+\lambda_\ell K$ intersect, 
any two of the compact intervals $\fhi(v_k+\lambda_kK)$ and 
$\fhi(v_\ell+\lambda_\ell K)$ intersect in $\bR$.
By Helly's Theorem in $\bR$, there exists 
$\alpha\in\bigcap_{t=1}^n\fhi(v_t+\lambda_t K)$.
Since $\fhi(v_i+\lambda_i K)=[\fhi(v_i)-\lambda_i,\fhi(v_i)+\lambda_i]$ and 
$\fhi(v_j+\lambda_j K)=[\fhi(v_j)-\lambda_j,\fhi(v_j)+\lambda_j]$, we have
\[ \fhi(v_j)-\lambda_j\leq \alpha\leq\fhi(v_i)+\lambda_i. \]
By the Minkowski property,
\[ \fhi(v_j-v_i)=\norm{v_j-v_i}_K\geq\max\{\lambda_i,\lambda_j\}.\]
It follows that
\begin{equation}\label{eq2}
\fhi(v_i)\leq \alpha\leq \fhi(v_j).
\end{equation}
We set $f=(\fhi,-\alpha)\in(\Red\times\bR)^*$, that is, define $f(x)=\fhi(v)-\alpha\mu$, where 
$x=(v,\mu)\in\Red\times\bR$.
We show that $f(x_j-x_i)\geq 1$, and $\abs{f(x_k)}\leq 1$ for all 
$k\in\{1,\dots,n\}$.
This will show that \eqref{eq:eq0} is satisfied, which will finish 
the proof.
\begin{align*}
f(x_j-x_i) &= \fhi(\lambda_j^{-1} v_j-\lambda_i^{-1} v_i) - 
\alpha(\lambda_j^{-1}-\lambda_i^{-1})\\
&= \frac{\fhi(v_j)-\alpha}{\lambda_j} + \frac{\alpha-\fhi(v_i)}{\lambda_i} \\
&\stackrel{\eqref{eq2}}{\geq} \frac{\fhi(v_j) - \alpha + \alpha - 
\fhi(v_i)}{\max\{\lambda_i,\lambda_j\}}\\ 
&= \frac{\norm{v_j-v_i}_K}{\max\{\lambda_i,\lambda_j\}} \geq 1.
\end{align*}
Since $\alpha\in\fhi(v_k+\lambda_k K)$, there exists $x\in K$ such that 
$\fhi(v_k+\lambda_k x)=\alpha$.
Therefore,
\[ \abs{f(x_k)} = \abs{\fhi(\lambda_k^{-1}v_k)-\alpha\lambda_k^{-1}} = 
\abs{\fhi(x)}\leq\norm{x}_K\leq 1. \qedhere\]
\qed

\section{Proof of Theorem~\ref{thm:sasha2}}
The following proof is very similar to the proof of Theorem~\ref{thm:lifting}.

Without loss of generality, $\min_i\lambda_i=1$.
Denote the unit ball of $\norm{\cdot}$ by $K$.
Let $x_i=(\lambda_i^{-1}v_i,\lambda_i^{-1})\in\bR^d\times\bR$, $i=0,\dots,n-1$.
Let $N\geq 1$, to be fixed later.
For each $m=0,\dots,N$, let 
\[ X_m = \setbuilder{x_i}{\lfloor N\log_2\lambda_i\rfloor\equiv m\pmod{N+1}} .\]
Then $X_0,\dots,X_N$ partition $\{x_0,\dots,x_{n-1}\}$ into $N+1$ parts.
Fix $i,j\in X_m$ such that $0\leq i<j < n$.
We will find a linear $f\colon\bR^d\times\bR\to\bR$ such that \eqref{eq:eq1} is satisfied for all $x_k\in X_m$ and $\lambda=2-2^{1/N}$.
Let $\fhi\colon\bR^d\to\bR$ be a linear functional such that $\fhi(x)\leq\norm{x}$ for all $x\in\bR^d$ and
\begin{equation}\label{2.0}
\fhi(x_j-x_i)=\norm{v_j-v_i}=\lambda_i.
\end{equation}
(Thus, $\fhi^{-1}(1)$ is a hyperplane that supports $K$ at $\norm{v_j-v_i}_K^{-1}(v_j-v_i)$.)

Since any two homothets $v_k+\lambda_k K$ and $v_\ell+\lambda_\ell K$ intersect in their interiors, any two of the open intervals $\fhi(v_k+\lambda_k\interior K)$ and $\fhi(v_\ell+\lambda_\ell\interior K)$ intersect in $\bR$.
By Helly's Theorem in $\bR$, there exists $\alpha\in\bigcap_{t=1}^n\fhi(v_t+\lambda_t\interior  K)$.
Since $\fhi(v_i+\lambda_i\interior K)=(\fhi(v_i)-\lambda_i,\fhi(v_i)+\lambda_i)$ and $\fhi(v_j+\lambda_j\interior K)=(\fhi(v_j)-\lambda_j,\fhi(v_j)+\lambda_j)$, we have
\[ \fhi(v_j)-\lambda_j < \alpha < \fhi(v_i)+\lambda_i. \]
By \eqref{2.0}, we can rewrite this as
\begin{equation}\label{2.1}
-\lambda_i < \fhi(v_i)-\alpha<\lambda_j-\lambda_i.
\end{equation}
We set $f=(\fhi,-\alpha)\in(\bR^d\times\bR)^*$, that is, for $x=(v,\mu)\in\bR^d\times\bR$, we let $f(x)=\fhi(v)-\alpha\mu$.
It remains to show that $f(x_j-x_i) > 2-2^{1/N}$, and $\abs{f(x_k)}\leq 1$ for all $k\in\{0,\dots,n\}$, since this will show that \eqref{eq:eq1} is satisfied with $\lambda=2-2^{1/N}$.
By applying Theorems~\ref{thm:equivalence} and 
\ref{thm:closedonesidedbound} with $\lambda=2/(2-2^{1/N})=2+\frac{\log 4}{N}+O(N^{-2})$, $K=(\lambda-1)L$ and $\alpha=1/(\lambda-1)=2^{1-1/N}-1$, we  obtain $\card{X_m}\leq (1+\lambda/2)(1+\lambda)^d$, and
it follows that \[n\leq (N+1)(1+\lambda/2)(1+\lambda)^d.\]
If we choose $N=d$, we obtain $\lambda=2+\frac{\log 4}{d} + O(d^{-2})$ and $n = 3^dO(d)$, which would finish the proof.

By definition of $X_m$,
\[ \lfloor N\log_2\lambda_j\rfloor-\lfloor N\log_2\lambda_i\rfloor = kN\quad\text{for some $k\in\Ze$.}\]
If $k\geq 1$, then $N\log_2\lambda_j-N\log_2\lambda_i>N$, hence $\lambda_j/\lambda_i>2$.
However, we also have
\[ \lambda_i=\norm{v_i-v_j}\geq\norm{v_j-v_n}-\norm{v_n-v_i}=\lambda_j-\lambda_i,\]
a contradiction.
Therefore, $k\leq 0$, that is,
$\lfloor N\log_2\lambda_j\rfloor-\lfloor N\log_2\lambda_i\rfloor \leq 0$.
This gives $N\log_2\lambda_j-N\log_2\lambda_i < 1$ and
\begin{equation}\label{2.2}
\frac{\lambda_j}{\lambda_i} < 2^{1/N}.
\end{equation}
It follows that
\begin{align*}
f(x_j-x_i) &= \fhi(\lambda_j^{-1} v_j-\lambda_i^{-1} v_i) - \alpha(\lambda_j^{-1}-\lambda_i^{-1})\\
&= \frac{\fhi(v_j)-\alpha}{\lambda_j} + \frac{\alpha-\fhi(v_i)}{\lambda_i} \\
&= \frac{\fhi(v_i)+\lambda_i-\alpha}{\lambda_j} + \frac{\alpha-\fhi(v_i)}{\lambda_i} \\
&\stackrel{\eqref{2.1}, \eqref{2.2}}{>} \frac{2^{-1/N}(\fhi(v_i)+\lambda_i - \alpha) + \alpha - \fhi(v_i)}{\lambda_i}\\ 
&= 2^{-1/N} + \frac{(1-2^{-1/N})(\alpha-\fhi(v_i))}{\lambda_i}\\
&\stackrel{\eqref{2.1}}{>} 2^{-1/N} +\frac{(1-2^{-1/N})(\lambda_i-\lambda_j)}{\lambda_i}\\ 
&= 1 - (1-2^{-1/N})\frac{\lambda_j}{\lambda_i}\\
&\stackrel{\eqref{2.1}}{>} 1 - (1-2^{-1/N})2^{1/N}\\ 
&= 2-2^{1/N}.
\end{align*}
Since $\alpha\in\fhi(v_k+\lambda_k K)$, there exists $x\in K$ such that $\fhi(v_k+\lambda_k x)=\alpha$.
Therefore,
\[ \abs{f(x_k)} = \abs{\fhi(\lambda_k^{-1}v_k)-\alpha\lambda_k^{-1}} = \abs{\fhi(x)}\leq\norm{x}_K\leq 1. \qedhere\]
\qed

\section{Proof of Theorem~\ref{thm:equivalence}}
Assume that \eqref{item:equivalncepoints} holds. Let $C:=\bigcap_{i\neq 
j}S_{ij}$ be the intersection of 
the $o$-symmetric slabs $S_{ij}:=\left\{p\in\ReD :  \abs{f_{ij}(p)}\leq 
\frac{\lambda}{2}\abs{f_{ij}(x_i)-f_{ij}(x_j)}\right\}$.
By assumption, $C\supseteq\{x_1,\ldots,x_n\}$.
For each $i\in[n]$, let $C_i:=\frac{\lambda x_i+C}{\lambda+1}$ be the 
homothetic 
copy of $C$ with center of homothety $x_i$, and of ratio $\frac{1}{\lambda+1}$.
It is an easy exercise that the $C_i$s are non-overlapping. Moreover, by the 
symmetry of $C$, we have $\frac{\lambda-1}{\lambda+1}x_i\in C_i\cap 
\frac{\lambda-1}{\lambda+1}C$. Thus, for $L:=\frac{1}{\lambda+1}C$, and 
$t_i:=\frac{\lambda}{\lambda+1}x_i$, \eqref{item:equivalncetranslates} holds as 
promised.

Next, assume that \eqref{item:equivalncetranslates} holds. Fix $i,j\in[n], 
i\neq j$. Since $L+t_i$ and $L+t_j$ are non-overlapping, there is a linear 
functional $f$ such that the two real intervals $s_i:=f(L+t_i)$ and 
$s_j:=f(L+t_i)$ do not overlap. These two intervals are of equal length, which 
we denote by $w$. Thus, we have
\begin{equation}\label{eq:wnotbig}
w\leq\abs{f(t_i)-f(t_j)}. 
\end{equation}

On the other hand, $s_k:=f(L+t_k)$ is also a real interval of length 
$w$ for any $k\in[n]$; and $s_0:=f((\lambda-1)L)$ is a 
0-symmetric real interval of length $(\lambda-1)w$, which intersects each 
$s_k$. 
Thus, for the center $f(t_k)$ of $s_k$, we have $\abs{f(t_k)}\leq 
\frac{(\lambda-1)w}{2}+\frac{w}{2}=\frac{\lambda w}{2}$. Now,
\eqref{eq:wnotbig} yields $\abs{f(t_k)}\leq \frac{\lambda}{2}\abs{f(t_i)-f(t_j)}$.
Thus, we may set $f_{ij}:=f$. This argument is valid for any $i$ and $j$, thus, 
with $x_i:=t_i$, we obtain \eqref{item:equivalncepoints}.

\section{Proof of Theorem~\ref{thm:closedonesidedbound}}
The proof is an almost verbatim copy of the proof of Theorem~1.5 of 
\cite{LN09}. There are two points of difference, which we will note.

We recall Lemma~3.1. of \cite{LN09}, which is a slightly more general version 
of the Lemma of \cite{BB03}.

\begin{lem}\label{lem:monotonicity}
	Let $f$ be a function on $[0,1]$ with the properties $f(0)\geq 0$, $f$ 
is positive and monotone increasing on $(0,1]$,
	and $f(x) = (g(x))^k$ for some concave function $g$ and $k > 0$. Then
	\[
		F(y) := \frac{1}{f(y)} \Int_0^y f(x) \dx
	\]
	is strictly increasing on $(0,1]$.
\end{lem}

\begin{proof}[Proof of Theorem~\ref{thm:closedonesidedbound}.]
Clearly, we may assume that $K$ is bounded, otherwise, by a projection, we can 
reduce the dimension.
Let $\alpha K + t_1$, $\alpha K+t_2, \ldots, \alpha K + t_n$ be pairwise 
non-overlapping translates of $\alpha K$ that intersect $K$. By the assumptions 
of the theorem, there is a non-zero vector $v\in\ReD$ such that $a_i :=\langle 
t_i,v \rangle \geq 0$ for $i\in[n]$. Set $h(x) := \{ p \in \ReD : \langle p,v 
\rangle = x\}$. Without loss of generality, we may assume that $h(-1)$ and 
$h(1)$ are supporting hyperplanes of $K$.

Clearly, $\alpha K + t_i$ is between $h(-\alpha)$ and $h(1+2\alpha)$,
and it is contained in $(1+2\alpha)K$, for $i\in[n]$.

\begin{equation}\label{eq:wholealpha}
	\Int_{-\alpha}^{1+2\alpha} \vol_{D-1}\left( \left(\bigcup_{i=1}^n 
\alpha K + t_i\right) \cap h(x) \right) \dx= n \alpha^D \vol_D(K).
\end{equation}

\begin{equation}\label{eq:upperbound1alpha}
	\Int_0^{1+2\alpha} \vol_{D-1}\left( \left(\bigcup_{i=1}^n \alpha K + 
t_i\right) \cap h(x) \right) \dx
\end{equation}
\[\leq	\Int_0^{1+2\alpha} \vol_{D-1}\left( (1+2\alpha)K \cap h(x) \right) \dx 
	= \frac{(1+2\alpha)^{d}}{2} \vol_D(K).
\]

We note that this was the first point of difference from the proof in 
\cite{LN09}: here, we do not subtract the contribution of $K$ in the total 
volume on the right hand side of the inequality.

Set $f(x) := \vol_{D-1}\left( \alpha K \cap h(x-\alpha) \right)$, and observe 
that the conditions of Lemma~\ref{lem:monotonicity}
are satisfied by $f$ (with $k=D-1$, by the Brunn--Minkowski inequality). We may 
assume that $a_1,\dots,a_m\leq \alpha < a_{m+1},\dots,a_n$.
By Lemma~\ref{lem:monotonicity},
\begin{eqnarray*}
	& & \Int_{-\alpha}^0 \vol_{D-1}\left( \left(\bigcup_{i=1}^n (\alpha K + 
t_i) \right) \cap h(x) \right) \dx =
	\sum_{i=1}^m \Int_0^{\alpha-a_i} f(x) \dx \\
	&\leq& \sum_{i=1}^m 
\Int_0^\alpha f(x) \dx \frac{f(\alpha-a_i)}{f(\alpha)} 
	=\frac{\alpha^d\vol_D(K)}{2f(\alpha)} \sum_{i=1}^m \vol_{D-1}\left( 
(\alpha K+t_i) \cap h(0) \right)\\
&=&
	 \frac{\alpha^d\vol_D(K)}{2f(\alpha)} \vol_{D-1}\left(\left( 
\bigcup_{i=1}^m (\alpha K+t_i)\right) \cap h(0) \right) \\
	 &\leq&	\frac{\alpha^d\vol_D(K)}{2f(\alpha)} \bigg[\vol_{D-1}\left( 
(1+2\alpha)K \cap h(0) \right)
\bigg] =
	\frac{\alpha(1+2\alpha)^{D-1}}{2}\vol_D(K).
\end{eqnarray*}

We note that this was the second point of difference from the proof in 
\cite{LN09}: again, the contribution of $K$ to the volume is not subtracted.

This inequality, combined with (\ref{eq:wholealpha}) and 
(\ref{eq:upperbound1alpha}), yields (\ref{eq:Halpha}).

\end{proof}

\providecommand{\bysame}{\leavevmode\hbox to3em{\hrulefill}\thinspace}
\providecommand{\MR}{\relax\ifhmode\unskip\space\fi MR }
\providecommand{\MRhref}[2]{%
  \href{http://www.ams.org/mathscinet-getitem?mr=#1}{#2}
}
\providecommand{\href}[2]{#2}

\end{document}